\documentclass[12pt]{amsart}
\usepackage{amsmath,amsthm,amsfonts,amssymb,latexsym,enumerate,url}

\begin{document}

\theoremstyle{plain}

\newtheorem{thm}{Theorem}[section]
\newtheorem{lem}[thm]{Lemma}
\newtheorem{conj}[thm]{Conjecture}
\newtheorem{pro}[thm]{Proposition}
\newtheorem{cor}[thm]{Corollary}
\newtheorem{que}[thm]{Question}
\newtheorem{rem}[thm]{Remark}
\newtheorem{defi}[thm]{Definition}

\newtheorem*{thmA}{THEOREM A}
\newtheorem*{CorB}{COROLLARY B}
\newtheorem*{thmC}{Theorem C}

\newtheorem*{thmAcl}{Main Theorem$^{*}$}
\newtheorem*{thmBcl}{Theorem B$^{*}$}

\numberwithin{equation}{section}

\newcommand{\Maxn}{\operatorname{Max_{\textbf{N}}}}
\newcommand{\Syl}{\operatorname{Syl}}
\newcommand{\dl}{\operatorname{dl}}
\newcommand{\Con}{\operatorname{Con}}
\newcommand{\cl}{\operatorname{cl}}
\newcommand{\Stab}{\operatorname{Stab}}
\newcommand{\Aut}{\operatorname{Aut}}
\newcommand{\Ker}{\operatorname{Ker}}
\newcommand{\fl}{\operatorname{fl}}
\newcommand{\Irr}{\operatorname{Irr}}
\newcommand{\SL}{\operatorname{SL}}
\newcommand{\GL}{\operatorname{GL}}
\newcommand{\SU}{\operatorname{SU}}
\newcommand{\GU}{\operatorname{GU}}
\newcommand{\Sp}{\operatorname{Sp}}
\newcommand{\Spin}{\operatorname{Spin}}
\newcommand{\PGL}{\operatorname{PGL}}
\newcommand{\PSL}{\operatorname{PSL}}
\newcommand{\PSU}{\operatorname{PSU}}
\newcommand{\PGU}{\operatorname{PGU}}
\newcommand{\FF}{\mathbb{F}}
\newcommand{\NN}{\mathbb{N}}
\newcommand{\N}{\mathbf{N}}
\newcommand{\C}{\mathbf{C}}
\newcommand{\OO}{\mathbf{O}}
\newcommand{\F}{\mathbf{F}}

\renewcommand{\labelenumi}{\upshape (\roman{enumi})}

\providecommand{\St}{\mathsf{St}}
\providecommand{\E}{\mathrm{E}}
\providecommand{\PSp}{\mathrm{PSp}}
\providecommand{\Sp}{\mathrm{Sp}}
\providecommand{\SO}{\mathrm{SO}}
\providecommand{\re}{\mathrm{Re}}

\def\irrp#1{{\rm Irr}_{p'}(#1)}

\def\Z{{\mathbb Z}}
\def\C{{\mathbb C}}
\def\Q{{\mathbb Q}}
\def\irr#1{{\rm Irr}(#1)}
\def\cent#1#2{{\bf C}_{#1}(#2)}
\def\syl#1#2{{\rm Syl}_#1(#2)}
\def\nor{\triangleleft\,}
\def\oh#1#2{{\bf O}_{#1}(#2)}
\def\Oh#1#2{{\bf O}^{#1}(#2)}
\def\zent#1{{\bf Z}(#1)}
\def\det#1{{\rm det}(#1)}
\def\ker#1{{\rm ker}(#1)}
\def\norm#1#2{{\bf N}_{#1}(#2)}
\def\alt#1{{\rm Alt}(#1)}
\renewcommand{\Im}{{\mathrm {Im}}}
\newcommand{\Ind}{{\mathrm {Ind}}}
\newcommand{\diag}{{\mathrm {diag}}}
\newcommand{\soc}{{\mathrm {soc}}}
\newcommand{\End}{{\mathrm {End}}}
\newcommand{\sol}{{\mathrm {sol}}}
\newcommand{\Hom}{{\mathrm {Hom}}}
\newcommand{\Mor}{{\mathrm {Mor}}}
\newcommand{\Mat}{{\mathrm {Mat}}}
\newcommand{\Tr}{{\mathrm {Tr}}}
\newcommand{\tr}{{\mathrm {tr}}}
\newcommand{\Gal}{{\it Gal}}
\newcommand{\Spec}{{\mathrm {Spec}}}
\newcommand{\ad}{{\mathrm {ad}}}
\newcommand{\Sym}{{\mathrm {Sym}}}
\newcommand{\Char}{{\mathrm {char}}}
\newcommand{\pr}{{\mathrm {pr}}}
\newcommand{\rad}{{\mathrm {rad}}}
\newcommand{\abel}{{\mathrm {abel}}}
\newcommand{\codim}{{\mathrm {codim}}}
\newcommand{\ind}{{\mathrm {ind}}}
\newcommand{\Res}{{\mathrm {Res}}}
\newcommand{\Ann}{{\mathrm {Ann}}}
\newcommand{\Ext}{{\mathrm {Ext}}}
\newcommand{\Alt}{{\mathrm {Alt}}}
\newcommand{\CC}{{\mathbb C}}
\newcommand{\ch}{{\mathcal C}}
\newcommand{\CB}{{\mathbf C}}
\newcommand{\RR}{{\mathbb R}}
\newcommand{\QQ}{{\mathbb Q}}
\newcommand{\ZZ}{{\mathbb Z}}
\newcommand{\NB}{{\mathbf N}}
\newcommand{\ZB}{{\mathbf Z}}
\newcommand{\OB}{{\mathbf O}}
\newcommand{\EE}{{\mathbb E}}
\newcommand{\PP}{{\mathbb P}}
\newcommand{\GC}{{\mathcal G}}
\newcommand{\HC}{{\mathcal H}}
\newcommand{\AC}{{\mathcal A}}
\newcommand{\SC}{{\mathcal S}}
\newcommand{\TC}{{\mathcal T}}
\newcommand{\CL}{{\mathcal C}}
\newcommand{\EC}{{\mathcal E}}
\newcommand{\Om}{\Omega}
\newcommand{\eps}{\epsilon}
\newcommand{\varep}{\varepsilon}
\newcommand{\al}{\alpha}
\newcommand{\chis}{\chi_{s}}
\newcommand{\sigmad}{\sigma^{*}}
\newcommand{\PA}{\boldsymbol{\alpha}}
\newcommand{\gam}{\gamma}
\newcommand{\lam}{\lambda}
\newcommand{\la}{\langle}
\newcommand{\ra}{\rangle}
\newcommand{\hs}{\hat{s}}
\newcommand{\htt}{\hat{t}}
\newcommand{\tn}{\hspace{0.5mm}^{t}\hspace*{-0.2mm}}
\newcommand{\ta}{\hspace{0.5mm}^{2}\hspace*{-0.2mm}}
\newcommand{\tb}{\hspace{0.5mm}^{3}\hspace*{-0.2mm}}
\def\skipa{\vspace{-1.5mm} & \vspace{-1.5mm} & \vspace{-1.5mm}\\}
\newcommand{\tw}[1]{{}^#1\!}
\renewcommand{\mod}{\bmod \,}

\title{Alternating sums over $\pi$-subgroups}
\author{Gabriel Navarro}
\address{Department of Mathematics,   Universitat de Val\`encia, 46100 Burjassot,
Val\`encia, Spain}
\email{gabriel@uv.es}
\author{Benjamin Sambale}
\address{Institut f\"ur Algebra, Zahlentheorie und Diskrete Mathematik, Leibniz Universit\"at Hannover, Welfengarten~1, 30167 Hannover, Germany}
\email{sambale@math.uni-hannover.de}

\thanks{
The research of the first author supported by Ministerio de Ciencia e Innovaci\'on PID2019-103854GB-I00 and FEDER
funds. The second author thanks the German Research Foundation (projects SA 2864/1-2 and SA 2864/3-1).}



\begin{abstract}
Dade's conjecture predicts that if $p$
is a prime, then the number of irreducible characters of a finite group of a given $p$-defect  is determined by local subgroups. In this paper we replace $p$ by a set of primes $\pi$ and prove a $\pi$-version of Dade's conjecture for $\pi$-separable groups. This extends the (known) $p$-solvable case of the original conjecture and relates to a $\pi$-version of Alperin's weight conjecture previously established by the authors.
\end{abstract}

\maketitle
 
\section{Introduction} 

One of the most general local-global counting conjecture for irreducible complex characters of finite groups is due to E. C. Dade~\cite{D}. 
For a finite group $G$, a prime $p$ and an integer $d>0$, the conjecture asserts that the number of irreducible characters of $G$ of $p$-defect $d$ can be computed by an alternating sum over chains of $p$-subgroups. (In this paper, we only deal with the group-wise ordinary conjecture; see \cite[Conjecture~9.25]{N}.)
Dade~\cite{D} already showed that his conjecture implies Alperin's weight conjecture. The first author has proved that McKay's conjecture is also a consequence of Dade's conjecture (see \cite[Theorem~9.27]{N}). Dade's conjecture is known to be true
for $p$-solvable groups by work of G.~R.~Robinson~\cite{R} (see also Turull~\cite{T17}), and a reduction of it to
simple groups has been recently conducted by B.~Sp\"ath~\cite{Sp}. 

In previous work by Isaacs--Navarro~\cite{IN} and the present authors~\cite{NS}, we have replaced $p$ by a set of primes $\pi$ in order to prove variants of Alperin's weight conjecture for $\pi$-separable groups. 
In this paper, we are interested in chains of $\pi$-subgroups and alternating sums: that is, we look for a $\pi$-version of Dade's conjecture and for possible applications. For instance: if $\ch(G)$ is the set of chains of $\pi$-subgroups of $G$,  $G_C$ is the stabilizer in $G$ of the chain $C$, and $k(G_C)$ is the number of conjugacy classes of $G_C$, does the number
\[\mu_\pi(G)= \sum_{C \in \ch(G)} (-1)^{|C|} \frac{|G_C|}{|G|}k(G_C)\]
have any group-theoretical meaning? (Notice that $\mu_\pi(G)$ is an integer, since the factor $|G_C|/|G|$ disappears when we sum over $G$-conjugacy classes of chains.)

If $\pi=\{p\}$, then the Alperin weight conjecture (with the Kn\"orr--Robinson~\cite{KR} reformulation) asserts that $\mu_p(G)$ is the number of $p$-defect zero characters of $G$. In particular, this is the case for $p$-solvable groups.
If $G={\rm PSL}_2(11)$ and $\pi=\{2,3\}$, say, then $\mu_\pi(G)=0$, while $G$ has 2
irreducible characters with $\pi$-defect zero (i.\,e. $p$-defect zero for every $p\in\pi$). So, whatever the meaning of $\mu_\pi(G)$ is,
certainly it is not the number of $\pi$-defect zero characters of $G$. 

For an integer $d\ge 1$, we let $k_d(G)$ to be the
number of irreducible characters $\chi \in \irr G$ 
such that $|G|_\pi=d\chi(1)_\pi$, where $n_\pi=\prod_{p \in \pi}n_p$, and $n_p$ is the largest
power of $p$ dividing the positive integer $n$. 
(Notice that this deviates slightly from the usual notation for $\pi=\{p\}$.)

Our main result is a natural generalization of Dade's conjecture for $p$-solvable groups:\footnote{Theorem~A was proposed as a conjecture in the second author's Oberwolfach talk in 2019.}

\begin{thmA}
Let $G$ be a $\pi$-separable group, and let $d>1$.
Then 
\[\sum_{C \in \ch(G)} (-1)^{|C|}|G_C|k_d(G_C)=0.\]
\end{thmA}

Unlike the original conjecture in the case $\pi=\{p\}$, we cannot restrict ourselves to so-called \emph{normal} chains in Theorem~A (see \cite[Theorem~9.16]{N}). In fact, $G=\sf{S}_3$ with $\pi=\{2,3\}$ is already a counterexample. (This is related to the fact that $\pi$-subgroups are not in
general  nilpotent!) For this reason,  the known proofs of the $p$-solvable case cited above cannot be carried over to $\pi$.
We will obtain Theorem~A as a special case of a more general projective statement with respect to normal subgroups.

\medskip

In this paper, let  $l(G)$  be  the number of conjugacy classes of $\pi'$-elements in $G$. Recall that $\chi \in \irr G$
has $\pi$-defect zero if $\chi(1)_p=|G|_p$ for all $p \in \pi$. The number of those characters is $k_1(G)$,
using  the notation above.

\begin{CorB}
Let $G$ be a $\pi$-separable group.
Then 
\[\sum_{C \in \ch(G)} (-1)^{|C|} \frac{|G_C|}{|G|}k(G_C)=\sum_{C \in \ch(G)} (-1)^{|C|} \frac{|G_C|}{|G|}l(G_C)\]
is the number of $\pi$-defect zero characters of $G$.
\end{CorB}

\section{Proofs}
We fix  a set of primes $\pi$ for the rest of the paper. If $G$ is a finite group,
 we consider chains $C$ of $\pi$-subgroups in $G$ of the form $1=P_0<P_1<\ldots<P_n$ where $n=0$ is allowed (the trivial chain). Let $|C|=n$ and let 
\[G_C=\N_G(P_0)\cap\ldots\cap\N_G(P_n)\]
be the stabilizer of $C$ in $G$. The set of all such chains of $G$ is denoted by $\ch(G)$. 

For a normal subgroup $N$ of $G$ and $\theta \in \irr N$, let $k_d(G|\theta)$ be the number of
irreducible characters $\chi$ of $G$ lying over $\theta$ with $|G|_\pi=d\chi(1)_\pi$.
We denote by $G_\theta$ the stabilizer of $\theta$ in $G$.  By the Clifford correspondence,
notice that 
\[k_d(G|\theta)=k_d(G_\theta|\theta).\]

We start with the following. 
\begin{lem}\label{funct}
Let $G$ be a finite group, and let $f$ be a real-valued function on the set of subgroups
of $G$. If $\OO_{\pi}(G)>1$, then
\[\sum_{C \in \ch(G)} (-1)^{|C|}|G_C|f(G_C)=0.\]
\end{lem}
\begin{proof}
Let $C:1=P_0<\ldots<P_n$ be a chain in $\ch(G)$. If $N=\OO_\pi(G)\nsubseteq P_n$, we obtain $C^*\in\ch(G)$ from $C$ by adding $NP_n$ at the end which is still a $\pi$-group. Otherwise let $N\subseteq P_k$ and $N\nsubseteq P_{k-1}$. If $P_{k-1}N=P_k$, then we delete $P_k$, otherwise we add $P_{k-1}N$ between $P_{k-1}$ and $P_k$. It is easy to see that in all cases $|C^*|=|C|\pm1$, $(C^*)^*=C$ and $G_C=G_{C^*}$. Hence, the map $C\mapsto C^*$ is a bijection on $\ch(G)$ such that
\begin{align*}
\sum_{C\in\ch(G)}(-1)^{|C|}|G_C|f(G_C)&=\sum_{C\in\ch(G)}(-1)^{|C^*|}|G_{C^*}|f(G_{C^{*}})\\
&=-\sum_{C\in\ch(G)}(-1)^{|C|}|G_C|f(G_C)=0.\qedhere
\end{align*}
\end{proof}

It is obvious that $G$ acts by conjugation on $\ch(G)$. The set of $G$-orbits is denoted by $\ch(G)/G$ in the following.
If $K \nor G$, notice that $G$ also acts on $\ch(G/K)$.

\begin{lem}\label{chainsmod}
Let $G$ be  a finite group
 with a normal $\pi'$-subgroup $K$. Let $\overline{H}:=HK/K$ for $H\le G$. 
 \begin{enumerate}[(a)]
 \item
 The map
$\ch(G)\mapsto \ch(\overline{G})$ given by 
\[C:P_0<\ldots<P_n \mapsto\overline{C}:\overline{P_0}<\ldots<\overline{P_n}\]
induces a bijection $\ch(G)/G\to\ch(\overline{G})/\overline{G}$. 

\item
For $\overline{C}\in\ch(\overline{G})$,
we have that $\overline{G}_{\overline{C}}=G_{\overline{C}}/K=G_CK/K$.

\item
Let $f$ be a real-valued function on the set of subgroups of $G$ such that $f(H)=f(H^g)$ for all $H\le G$ and $g\in G$. Then
\[\sum_{C\in\ch(G)}(-1)^{|C|}|G_C|f(G_CK)=\sum_{\overline{C}\in\ch(\overline{G})}(-1)^{|\overline{C}|}|G_{\overline{C}}|f(G_{\overline{C}}).\]
\end{enumerate}
\end{lem}
\begin{proof}
First,  we notice that the map
$\ch(G)\to\ch(\overline{G})$ given by  $C\mapsto\overline{C}$ is surjective.
Indeed, suppose that $\overline{C}:1=U_0/K < \ldots <U_n/K$ is a chain of $\pi$-subgroups of
$G/K$. By the Schur--Zassenhaus theorem, we have that
$U_n=KP_n$ for some $\pi$-subgroup $P_n$
of $G$. Then $U_i=K(U_i \cap P_n)$, and therefore the chain
$1=U_0\cap P_n< \ldots < P_n$ maps to $\overline{C}$. 

If chains $C:P_0<\ldots<P_n$ and $D:Q_0<\ldots<Q_n$ are conjugate in $G$, then $\overline{C}$ and $\overline{D}$ are obviously conjugate in $\overline{G}$. Suppose conversely that $\overline{C}$ and $\overline{D}$ are $\overline{G}$-conjugate. Without loss of generality,
we may assume that  $P_iK=Q_iK$ for $i=0,\ldots,n$. Again by the Schur--Zassenhaus theorem (this time relying on the Feit--Thompson theorem), $P_n$ is conjugate to $Q_n$ by some $x\in K$. We still have $P_i^xK=Q_iK$ for $i=0,\ldots,n$. Since $P_i^x,Q_i\le Q_n$ it follows that $P_i^x=Q_i$ for $i=0,\ldots,n$. Hence, $C$ and $D$ are $G$-conjugate. 
This proves (a).

Suppose that $P_1<P_2$ are $\pi$-subgroups of $G$. We claim that
\[\norm G{P_1} K \cap \norm G{P_2}K=(\norm G{P_1} \cap \norm G{P_2})K.\]
If $x \in \norm G{P_1} K \cap \norm G{P_2}K$, then $P_2^x=P_2^k$ for some $k \in K$.
Therefore $xk^{-1} \in \norm G{P_2}\cap \norm G{P_1}K$. Since $P_1K \cap P_2=P_1$,
we have that $xk^{-1} \in \norm G{P_1}$, and therefore $x \in (\norm G{P_1} \cap \norm G{P_2})K$.
This proves the claim.

Suppose now that
$C:P_0<\ldots<P_n$ is a chain of $\pi$-subgroups of $G$.
By the Frattini argument, $\overline{\N_G(P_i)}=\N_{\overline{G}}(\overline{P_i})$ and therefore $\overline{G_C}=\overline{G}_{\overline{C}}$, using the last paragraph. 
Regarding the action of $G$ on $\ch(\overline{G})$ we also have $G_{\overline{C}}/K=\overline{G}_{\overline{C}}$.

Finally, we prove (c). The $G$-orbit of $C$ has size $|G:G_C|$, while the $\overline{G}$-orbit of $\overline{C}$ has size $|\overline{G}:\overline{G}_{\overline{C}}|=|G:G_{\overline{C}}|$. Let $C_1,\ldots,C_k$ be representatives for $\ch(G)/G$, so that
$\overline{C_1},\ldots,\overline{C_k}$ are representatives
for  $\ch(\overline{G})/\overline{G}$
. Then
\[\sum_{C\in\ch(G)}(-1)^{|C|}|G_C|f(G_CK)=|G|\sum_{i=1}^k (-1)^{|C_i|}f(G_{C_i}K)=\sum_{\overline{C}\in\ch(\overline{G})}(-1)^{\overline{C}}|G_{\overline{C}}|f(G_{\overline{C}}).\qedhere\]
\end{proof}

The deep part of our results comes from the ``above Glauberman--Isaacs correspondence" theory.
If $A$ is a solvable finite group, acting coprimely on $G$, recall that Glauberman
discovered a natural bijection $^*$ from ${\rm Irr}_A(G)$, the set of $A$-invariant irreducible
characters of $G$, and $\irr{\cent GA}$, the irreducible characters of the fixed-point subgroup.
The case where $A$ is a $p$-group is fundamental in the local/global counting conjectures.
If $A$ is not solvable, an important case in this paper, then $G$ has odd order by the Feit-Thompson
theorem. In this case, Isaacs~\cite{I} proved that there is also a natural bijection ${\rm Irr}_A(G) \rightarrow
\irr{\cent GA}$. T. R. Wolf~\cite{Wo} proved that both correspondences agree in the intersection
of their hypotheses. 

\begin{thm}\label{g-i}
Let $G$ be a finite group with a normal $\pi'$-subgroup $K$. Let $C\in\ch(G)$ with last subgroup $P_C=P_{|C|}$. Let $\tau \in \irr K$ be $P_C$-invariant, and let $\tau^* \in \irr{\cent K{P_C}}$ be its Glauberman--Isaacs
correspondent. Then
\[k_d(G_CK|\tau)=k_d(G_C|\tau^*)\]
for every integer $d$.
\end{thm}

\begin{proof}
Let $U=K(P_CG_C)$. Notice that $G_C \cap K=\cent K{P_C}$. 
Also, $KP_C \nor U$. Thus $U=K\norm U{P_C}$, by the Frattini argument and the
Schur--Zassenhaus theorem.  Also, 
\[\norm U{P_C}=\norm G{P_C}\cap (P_CG_C)K=
(P_CG_C)\norm K{P_C}=P_CG_C\cent K{P_C}.\]  
Since $G_C$ normalizes $P_C$,
we have that $G_C$ commutes with the
$P_C$-Glauberman--Isaacs correspondence.  In particular,
$$(G_C)_{\tau^*}=(G_C)_\tau \, .$$
Hence, by using the Clifford correspondence,
we may assume that $\tau$ is $G_C$-invariant (and therefore $U$-invariant) and
that $\tau^*$ is $G_C$-invariant too.

Now, we claim that the character triples
$(U,K, \tau)$ and $(\norm U{P_C}, \cent K{P_C}, \tau^*)$ are isomorphic.
If $P_C$ is solvable, this is a well-known fact which follows from the Dade--Puig theory. (A comprehensive proof is given in \cite{T08}.)  If $P_C$ is not solvable, then
$|K|$ is odd, by the Feit--Thompson theorem. Then the claim follows from
the theory developed by Isaacs in \cite{I}. (A proof is given in the last paragraphs of \cite{L}.)

Since the character triples
$(U,K, \tau)$ and $(\norm U{P_C}, \cent K{P_C}, \tau^*)$ are isomorphic, it follows
from the definition that the sub-triples $(G_CK, K, \tau)$ and $(G_C, \cent K{P_C}, \tau^*)$
are isomorphic too. This yields a bijection $\Irr(G_CK|\tau)\to\Irr(G_C|\tau^*)$, $\chi\mapsto\chi^*$ such that $\chi(1)/\tau(1)=\chi^*(1)/\tau^*(1)$ (see \cite[p. 87]{N}). In particular, $k_d(G_CK|\tau)=k_d(G_C|\tau^*)$ (if $d_\pi\ne d$, both numbers are $0$).
\end{proof}

Theorem~A is the special case $N=1$ of the following projective version.

\begin{thm}\label{main}
Let $G$ be a $\pi$-separable group with a normal $\pi'$-subgroup $N$. Let $\theta \in \irr N$ be $G$-invariant and $d>1$.
Then
\[\sum_{C \in \ch(G)} (-1)^{|C|}|G_C|k_d(G_CN|\theta)=0.\]
\end{thm}

\begin{proof} 
We may assume that $d_\pi=d$. 
We argue by induction on $|G:N|$. Let $\overline{G}=G/N$. By Lemma~\ref{chainsmod}, we may sum over $\overline{C}\in\ch(\overline{G})$ by replacing $G_C$ with $G_{\overline{C}}$. 
Recall that a character triple isomorphism $(G,N,\theta)\to(G^*,N^*,\theta^*)$ induces an isomorphism $\overline{G}\cong G^*/N^*$ and a bijection $\Irr(G|\theta)\to\Irr(G^*|\theta^*)$, $\chi\mapsto\chi^*$ such that $\chi(1)/\theta(1)=\chi^*(1)/\theta^*(1)$. Thus, $k_d(G_{\overline{C}}|\theta)=k_d(G^*_{\overline{C}^*}|\theta^*)$ and the numbers $|G_{\overline{C}}|$, $|G^*_{\overline{C}^*}|$ differ only by a factor independent of $C$. This allows us to replace $N$ by $N^*$. Using \cite[Corollary~5.9]{N}, we assume that $N$ is a central $\pi'$-subgroup in the following. 
Now using Lemma~\ref{chainsmod} in the opposite direction, we sum over $C\in\ch(G)$ again and note that $N\subseteq G_C$.
Thus, it suffices to show that 
\[\sum_{C \in \ch(G)} (-1)^{|C|}|G_C|k_d(G_C|\theta)=0.\]

By Lemma \ref{funct}, we may assume that $\oh \pi G=1$.
Let $K=\oh{\pi'} G$. If $K=N$, then $N=G$ by the Hall--Higman 1.2.3~lemma.
In this case, the theorem is correct because $d>1$. 
So we may assume that $K>N$. Let $P_C$ be the last member of $C\in\ch(G)$. Observe that $G_C\cap K=\cent K{P_C}$. 
Each $\psi\in\Irr(G_C|\theta)$ lies over some $\mu\in\Irr(\cent K{P_C}|\theta)$. But $\psi$ lies also over $\mu^g$ for every $g\in G_C$. Therefore,
\[
\sum_{C \in \ch(G)} (-1)^{|C|} {|G_C|}|k_d({G_C |\theta})=\sum_{C \in \ch(G)} (-1)^{|C|} {|G_C|} 
\Bigl(\sum_{\mu \in \irr{\cent K{P_C}|\theta}}\frac{ k_d(G_C | \mu )}{ |G_C:G_{C, \mu}|}\Bigr)\]
where $G_{C,\mu}=G_C\cap G_\mu$. 
According to Theorem~\ref{g-i}, we replace $\irr{\cent K{P_C}|\theta}$ by $\Irr_{P_C}(K|\theta)$ and $k_d(G_C | \mu )$ by $k_d(G_CK |\mu)$. By the Clifford correspondence, $k_d(G_CK |\mu)=k_d(G_{C,\mu}K|\mu)$.
Moreover, $\mu \in \Irr_{P_C}(K|\theta)$ implies $P_C\le G_\mu$. Thus, for a fixed $\mu$ we only need to consider chains in $G_\mu$. Hence,
\[\sum_{C \in \ch(G)} (-1)^{|C|} {|G_C|}|k_d({G_C |\theta})=\sum_{\mu \in \Irr(K|\theta)}
\Bigl(\sum_{C \in \ch(G_\mu)} (-1)^{|C|} |G_{C, \mu}| k_d(G_{C,\mu}K | \mu )\Bigr).\]
Since $|G_\mu:K|<|G:N|$, the inner sum vanishes for every $\mu$ by induction. Hence, we are done.
\end{proof}

Finally, we come to our second result.

\begin{proof}[Proof of Corollary~B]
Let $C:P_0<\ldots<P_n$ in $\ch(G)$ such that $n>0$. Then $P_1\unlhd G_C$. Let $\chi\in\Irr(G_C)$ and $\theta\in\Irr(P_1)$ under $\chi$. By Clifford theory, $\chi(1)/\theta(1)$ divides $|G_C/P_1|$ (see \cite[Theorem~5.12]{N}). Since $\theta(1)<|P_1|$, it follows that $\chi(1)_\pi<|G_C|_\pi$ and $k_1(G_C)=0$. Summing over $d\ge 1$ in Theorem~A yields
\[\sum_{C \in \ch(G)} (-1)^{|C|} \frac{|G_C|}{|G|}k(G_C)=k_1(G).\]
The second equality follows from a straight-forward generalization of the Kn\"orr--Robinson argument. In fact, the proofs of \cite[9.18--9.23]{N} go through word by word (replacing $p$ by $\pi$, of course). 
\end{proof} 

Given the proof above, we take the opportunity to point out that a theorem of Webb~\cite{We} (see also \cite[Corollary~9.20]{N}) remains true in the $\pi$-setting:

\begin{thm}
Let $G$ be an arbitrary finite group, and let $\pi$ be a set of primes.
Then the generalized character
\[\sum_{C\in\ch(G)}(-1)^{|C|}|G_C|(1_{G_C})^G\]
vanishes on all elements of $G$ whose order is divisible by a prime in $\pi$.
\end{thm}

Unlike the case where $\pi=\{p\}$, Alperin's weight conjecture cannot be deduced from Corollary~B. As a matter of fact, for $\pi$-separable groups $G$, we can prove that there is no formula of the form
\[l(G)=\sum_{P}\alpha_Pk_1(\norm{G}{P}/P)\]
where $P$ runs through the $G$-conjugacy classes of $\pi$-subgroups and the coefficients $\alpha_P\ge 0$ depend only on the isomorphism type of $P$.

\medskip

It is interesting to speculate on variations of Theorem A, that is
projective versions of Dade's conjecture, that might be even true
for arbitrary normal subgroups of any finite group $G$, whenever $\pi=\{p\}$.
Outside $\pi$-separable groups, we do not know what is the meaning, if any, of
the number $\mu_\pi(G)$. In fact, this number can even be negative in groups
with a Hall $\pi$-subgroup. We have not attempted a block version of Theorem A.
Although   $\pi$-block theory is well-developed in $\pi$-separable
groups (see \cite{Sl}, for instance), Brauer's block induction does not behave well if Hall $\pi$-subgroups
are not nilpotent. 

\medskip

Computations with chains are almost impossible to do by hand. The results of this paper would not have been discovered without the help of \cite{GAP}.


\begin{thebibliography}{ABCDE}
  
\bibitem[D]{D}
E. C. Dade, \textit{Counting characters in blocks, I}, Invent. Math. \textbf{109}, (1992), 187--210.

\bibitem[GAP]{GAP}
The GAP group, `{\it {\sf GAP} - groups, algorithms, and
programming}', Version 4.11.0, 
2020, \url{http://www.gap-system.org}.

\bibitem[I]{I}
I. M.  Isaacs, \textit{Characters of solvable and symplectic groups},  Amer. J. Math. {\bf 85} (1973),  594--635.

\bibitem[IN]{IN}
I. M.  Isaacs, G. Navarro, \textit{Weights and vertices for characters of 
$\pi$-separable groups},  J. Algebra {\bf 177} (1995),  339--366. 

\bibitem[KR]{KR}
R. Kn\"orr, G. R. Robinson, \textit{Some remarks on a conjecture of Alperin}, J. London Math. Soc. \textbf{39} (1989), 48--60.

\bibitem[L]{L}
M. L. Lewis, \textit{Characters, coprime actions,
and operator groups},  Arch. Math. {\bf 69} (1997),  455--460. 

\bibitem[N]{N}
G. Navarro, \textit{Character theory and the {M}c{K}ay conjecture}, Cambridge
  Studies in Advanced Mathematics, \textbf{175}, Cambridge University Press,
  Cambridge, 2018.

\bibitem[NS]{NS}
G. Navarro, B. Sambale, \textit{Weights and nilpotent subgroups}, 
Int. Math. Res. Not. \textbf{2021} (2021), 2526--2538.

\bibitem[R]{R}
G.~R. Robinson, \textit{Dade's projective conjecture for {$p$}-solvable
  groups}, J. Algebra \textbf{229} (2000), 234--248.

\bibitem[Sl]{Sl}
M.~C. Slattery, \textit{Pi-blocks of pi-separable groups. {I}}, J. Algebra
  \textbf{102} (1986), 60--77.

\bibitem[Sp]{Sp}
B. Sp\"ath, \textit{A reduction theorem for Dade's projective conjecture}, J. Eur. Math. Soc. \textbf{19} (2017), 1071--1126.

\bibitem[T08]{T08}
A. Turull, \textit{Above the Glauberman correspondence},
Adv. in Math. {\bf 217} (2008), 2170--2205.

\bibitem[T17]{T17}
A. Turull, \textit{Refinements of Dade's Projective Conjecture for p-solvable groups},
J. Algebra {\bf 474} (2017), 424--465.

\bibitem[We]{We}
P. J. Webb, \textit{Subgroup complexes}, Proc. Sympos. Pure Math. \textbf{47} (1987), 349–365.

\bibitem[Wo]{Wo}
  T. R. Wolf, \textit{Character correspondences in solvable groups},
  {\sl  Illinois J. Math.} {\bf 22}  (1978), 327--340.

\end{thebibliography}
\end{document}